\documentclass[11pt,leqno]{article}
\usepackage{amsfonts}
\usepackage{amsmath}
\usepackage{graphicx}
\usepackage{color}

\newtheorem{theorem}{Theorem}

\newtheorem{conjecture}[theorem]{Conjecture}
\newtheorem{corollary}[theorem]{Corollary}

\newtheorem{definition}[theorem]{Definition}

\newtheorem{lemma}[theorem]{Lemma}

\newtheorem{remark}[theorem]{Remark}

\newenvironment{proof}[1][Proof]{\noindent\textbf{#1.} }{\ \rule{0.5em}{0.5em}}

\begin{document}

\title{\vspace{-10mm} Intrinsic Supermoothness}
\author{
 Boris Shekhtman\footnote{University of South Florida, 4202 East Fowler Ave, CMC342, Tampa, FL 33620, shekhtma@usf.edu}\\
 Tatyana Sorokina\footnote{Towson University, 7800 York Road, Towson, MD 21252, tsorokina@towson.edu}}
 \date{}
\maketitle

\begin{abstract}
The phenomenon, known as``supersmoothness" was first
observed for bivariate splines and attributed to the polynomial nature of
splines. Using only standard tools from multivatiate calculus, we show that
if we continuously glue two smooth functions along a curve with a
``corner", the resulting continuous function must be
differentiable at the corner, as if to compensate for the singularity of the
curve. Moreover, locally, this property, we call supersmoothness,
characterizes non-smooth curves. We also generalize this phenomenon to
higher order derivatives. In particular, this shows that supersmoothness has
little to do with properties of polynomials.
\end{abstract}

\vskip10pt

\section{Introduction}

In this short article we address supersmoothness: a phenomenon where under
certain circumstances continuity of a function of two variables implies its
differentiability at a point or, consequently, differentiability of a
bivariate function implies its higher order differentiability at a point.
Supersmoothness was first observed for a particular class of piecewise
bivariate polynomial functions, called splines, by Farin in~\cite{Farin}. He
considered a triangle $\Delta $ partitioned into three subtriangles $\Delta
_{1},\Delta _{2}$ and $\Delta _{3}$ as shown in Figure~\ref{circles}. 
\begin{figure}[th]
\begin{minipage}[b]{0.45\linewidth}
\centering
 \includegraphics[keepaspectratio=true, clip=true, scale=3.3, width=100mm, height=100mm, trim=60mm 210mm 0mm 20mm]{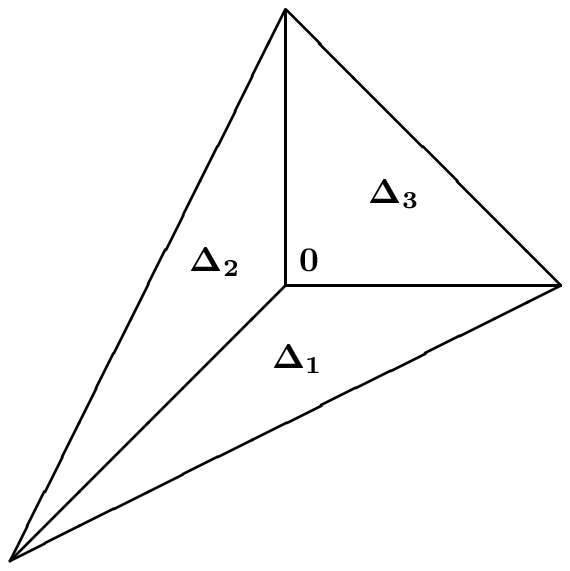}
 \caption{First example}
\label{circles}
 \end{minipage}
\hspace{1.1cm} 
\begin{minipage}[b]{0.45\linewidth}
 \centering
 \includegraphics[keepaspectratio=true, clip=true, scale=3.3, width=110mm, height=110mm, trim=80mm 210mm 0mm 20mm]{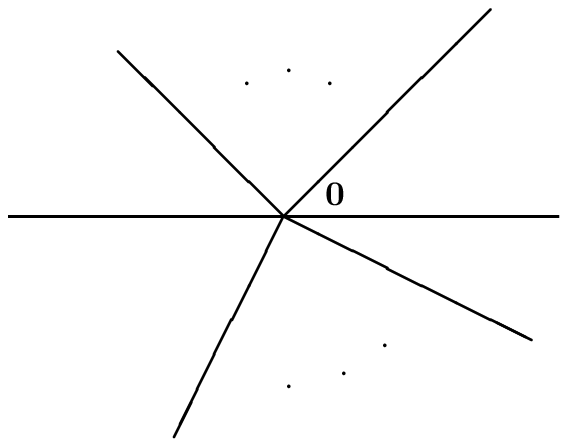}
\caption{Collinearity matters}
\label{counter}
\end{minipage}
\end{figure}
A spline $F$ on this triangulation of $\Delta $ is a function of two
variables such that for each $i=1,2,3$, the restriction $F\mid _{\Delta
_{i}}=f_{i}$ a polynomial. Farin proved that if the spline $F$ is
differentiable of order~$n$, then it has all $(n+1)$st order partial
derivatives at the origin $\mathbf{0}:=(0,0)$. That is for all $n\geq 1$: 
\begin{equation}
F\in C^{n}(\Delta )\Rightarrow F\in C^{n+1}(\mathbf{0}).  \label{1}
\end{equation}
Supersmoothness of splines was observed for trivariate splines in~\cite{Alf},
and studied in general in~\cite{me}. This phenomenon has been attributed
to the polynomial nature of splines. Recently, while flying on Delta, the
authors were struck by the similarity of the emblem of the airline and
Figure~\ref{circles}. This lead us to the subject of this paper. 

In the next section we will demonstrate that basic supersmoothness is a
rather general property of non-smooth curves, not just polynomials. Loosely
speaking: if we want to continuously glue two smooth bivariate  functions along a curve with
a ``corner" at a point $P$, the resulting continuous function
must be differentiable at $P$, as if to compensate for the singularity of
the curve. Moreover, locally, supersmoothness characterizes non-smooth
curves.

In Section 3 we address another peculiarity of supersmoothness. We
first show that property~(\ref{1}) holds for all smooth functions defined
over a partition of~$\mathbb{R}^{2}$ by $n+2$ non-collinear rays emanating
from the origin, $n\geq 0$. The assumption that the rays are not collinear
is significant. If just two of the rays are parallel the phenomenon of
automatic supersmoothness disappears alltogether.
 This can be seen on the following simple example. Consider the
partition of~$\mathbb{R}^{2}$ by the $x$-axis. For any $n\geq 0$, let $f(x,y)$
be equal to $y^{n+1}$ on the upper half plane and zero on the lower one. We
now can add any $n$ rays emanating from the origin but not along the $x$-axis.
 This will form a partition of~$\mathbb{R}^{2}$ by $n+2$ rays as in Figure~\ref{counter}. Then 
$f$ has all derivatives of order $n$, yet $f\notin C^{n+1}(\mathbf{0})$.

We note that all the proofs in this article are simple and should be accessible
to undergraduate students familiar with the basics of multivariate calculus, such as in~\cite{calc}.

\section{Gluing functions along a curve}

In this section we will show that a version of supersmoothness occurs when
we glue two differentiable functions along a curve with sharp corners as in
Figure~\ref{3}. Namely, we will show that the resulting piecewise function
is differentiable at every sharp corner of the curve. To some extent this
property of supersmoothness characterizes curves with sharp corners.

\begin{figure}[!ht]
\begin{minipage}[b]{0.45\linewidth}
\centering
 \includegraphics[keepaspectratio=true, clip=true, scale=3.3, width=230mm, height=230mm, trim=95mm 230mm 0mm 40mm]{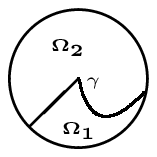}
 \caption{Curves for gluing}
\label{3}
 \end{minipage}
\hspace{1.1cm} 
\begin{minipage}[b]{0.45\linewidth}
 \centering
 \includegraphics[keepaspectratio=true, clip=true, scale=3.3, width=110mm, height=110mm, trim=80mm 210mm 0mm 20mm]{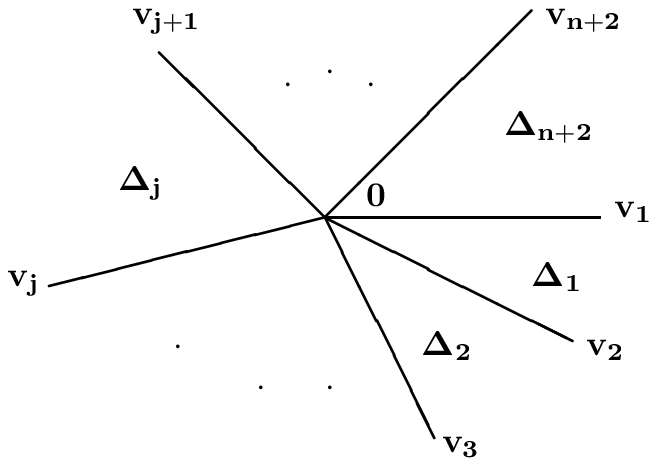}
\caption{Supersmoothness of higher derivatives}
\label{3.2}
\end{minipage}
\end{figure}

In contrast with most of the research on curves in both analysis and
differential geometry, we are interested in non-smooth curves. While
regularity of a curve is defined globally, non-smoothness has to be
localized to a point $P:=(x_{0},y_{0})$. Recall (cf.~\cite{calc}) that a curve 
\begin{equation*}
\gamma (t):=(u(t),v(t)):~[a,b]\rightarrow \mathbb{R}^{2}~
\end{equation*}
is regular if $\gamma $ is differentiable and $\gamma ^{\prime }(t)\neq 0$
for all $t$. In particular, for every point $t$ at least one of the
derivatives, say $u^{\prime }(t)\neq 0$. Hence the function $u(t)$ is
invertible in a neighborhood of that point. Setting $u(t)=x$ we have $t=u^{-1}(x)$ 
and we can reparametrize a portion of the curve as $(x,f(x))$
where $f:=v\circ u^{-1}$.

To keep this article within reach of calculus students we 
limit our considerations to non-self-intersecting curves and
adopt the following, intuitively clear version of ``local
smoothness". Let $\gamma$ be the trace of a continuous
non-self-intersecting curve $\gamma (t):~[a,b]\rightarrow \mathbb{R}^{2}$, also known as a Jordan arc. 
Without loss of generality assume that $\gamma(0)=P$ and $a<0<b.$

We shall say that $\gamma $ is smooth at $P$ if $\gamma$ can be
represented as a graph of a continuously differentiable function in some
neighborhood of $P$. More precisely,

\begin{definition}
\label{pointsmooth} The trace of a Jordan arc $\gamma $ is smooth at a point $P$  
if there exist open intervals $I$, $J$ and a function $f\in C^{1}(I)$ such that 
\begin{equation*}
P=(x_{0},f(x_{0}))\in I\times J,~~\hbox{and}~~\gamma \cap
(I\times J)=\{(x,f(x)),~~x\in I\}.
\end{equation*}
\end{definition}

\begin{theorem}
\label{newdirect} The trace of a Jordan arc $\gamma$ is smooth at $P$ if and
only if there exists a neighborhood $U$ of $P$ and a function $h$
continuously differentiable on $U$ such that 
\begin{equation}  \label{nd}
h(x,y)=0~~\hbox{if}~~(x,y)\in \gamma \cap U,~~\hbox{and}~~ \nabla h(P) \neq 
\mathbf{0}.
\end{equation}
\end{theorem}

\begin{proof}
If $\gamma$ is smooth at $P$ we use the neighborhood $I\times J$ and the 
$C^1$-continuous function $f$ from Definition~\ref{pointsmooth} to construct 
$h(x,y):=y-f(x)$. Clearly, $h$ satisfies all the desirable properties.
Conversely, without loss of generality, assume $P=(0,0)$ and let $h$ be a 
$C^1$-continuous function on some neighborhood~$U$ of~$P$, such that $h$
vanishes on~$\gamma$, and $h_y(P)\neq 0$. Then by the Implicit Function
Theorem, there exist open intervals $I_1$ and $J_1$ and a $C^1$-continuous
function $f$ such that 
\begin{equation*}
h(x,y)=0,~~(x,y)\in ~I_1\times J_1~~\hbox{iff}~~y=f(x),~~x\in I_1,~~y\in J_1.
\end{equation*}
We can assume that $I_1\times J_1\subset U$, which implies that 
\begin{equation*}
\gamma\cap (I_1\times J_1)\subseteq \{(x,f(x)),~~x\in I_1\}.
\end{equation*}
We now need to show that there exist perhaps other intervals $I\subseteq
I_1$ and $J\subseteq J_1$ such that $\gamma$ coincides with the graph of $f$
in $I\times J$. To this end, we consider the inverse image 
$\gamma^{-1}(I_1\times J_1)$, which is an open set in $[a,b]$ containing zero.
Thus, there exists $c>0$ such that $\gamma(t):=(u(t),v(t))$ maps $(-c,c)$
into $\gamma\cap (I_1\times J_1)$. We observe that if $u(c/2)=0$, then 
$v(c/2)$ also vanishes since $f$ is a function passing through $(0,0)$. Then
we have $\gamma(c/2)=\gamma(0)$ which contradicts the assumption that $\gamma
$ has no self-intersections. Thus, neither $u(c/2)$ nor $u(-c/2)$ vanish.
Without loss of generality we can assume $u(c/2)>0$. Then $u(-c/2)$ must be
negative. Otherwise either $\gamma$ has a self-intersection or $f$ is not a
function. Since $\gamma(t)$ is continuous, its trace from $t=-c/2$ to $t=c/2$
must coincide with the graph of $f$ from $u(-c/2)<0$ to $u(c/2)>0$. Thus,
for $I:=\big(u(-c/2),u(c/2)\big)$, and $J:=J_1$, the function $f$ satisfies
Definition~\ref{pointsmooth}.
\end{proof}

As a corollary we obtain the promised result on supersmoothness:

\begin{theorem}
\label{direct} Let $\gamma \subset \mathbb{R}^2$ be the trace of a Jordan
arc that divides the open disk $\Omega $ into two subsets $\Omega _{1}$ and 
$\Omega _{2}$ as in Figure~\ref{3}. Further assume that $\gamma$ is not
smooth at $P\in \gamma$. Let $f_1,f_2$ be $C^{1}$ functions on $\Omega $
continuously glued along $\gamma$, that is, let 
\begin{equation}
F(x,y):=\left\{ 
\begin{array}{ccc}
\label{2.7} f_1(x,y) & \text{if} & (x,y)\in \Omega _{1} \\ 
f_2(x,y) & \text{if} & (x,y)\in \Omega _{2}
\end{array}
\right.
\end{equation}
be a continuous function on $\Omega $. Then the piecewise function $F$ is
differentiable at $P$, that is, 
\begin{equation}  \label{2.8}
\nabla f_1(P)=\nabla f_2(P).
\end{equation}
\end{theorem}

\begin{proof}
Consider a $C^{1}$ function $h=f_1-f_2$. The fact that $f_1$ and $f_2$ are
continuously glued along $\gamma$ means that $h(\gamma)=0$ and by Theorem~\ref{newdirect} 
\begin{equation*}
0=\nabla h(P)=\nabla f_1(P)-\nabla f_2(P). 
\end{equation*}
Thus, $\nabla f_1(P)=\nabla f_2(P)$, and the proof is complete.
\end{proof}

A partial converse of Theorem~\ref{direct} holds true in the following sense:

\begin{theorem}
\label{partial} Let $\gamma \subset \mathbb{R}^{2}$ be the trace of a Jordan
arc that divides the open disk $\Omega $ into two subsets  $\Omega _{1}$ and 
$\Omega _{2}$. Assume that $\gamma $ is smooth at a point $P\in \gamma $.
Then there exists a neighborhood $U$ of $P$ and and two differentiable
functions $f_1,g_1\in C^{1}(U)$ such that the function 
\begin{equation}
F(x,y):=\left\{ 
\begin{array}{ccc}
\label{2.1} f_1(x,y) & \text{if} & (x,y)\in \Omega _{1} \\ 
f_2(x,y) & \text{if} & (x,y)\in \Omega _{2}
\end{array}
\right.
\end{equation}
is not differentiable at $P$.
\end{theorem}

\begin{proof}
Let $U$ and $h$ be chosen as in Theorem~\ref{newdirect}, i.e., satisfying
conditions~(\ref{nd}). Let $f_1(x,y):=h(x,y)$ and $f_2(x,y)\equiv 0$. Then,
since $h(\gamma \cap U)=0$, the function $F$ defined by~(\ref{2.1}) is
continuous and not differentiable at $P$ because 
$\nabla f_2(P)=\mathbf{0}\neq \nabla f_1(P)$.
\end{proof}

Theorem~\ref{partial} provides only a partial converse of Theorem~\ref{direct} 
because the function $F$ is defined locally, in a neighborhood $U$
of $P$, and not on all of~$\Omega$. We believe that the global
version of this theorem also holds and end this section with a conjecture.

\begin{conjecture}
\label{conj} Let $\gamma \subset \mathbb{R}^2$ be a continuous curve that
divides an open disk~$\Omega $ centered at $P$ into two subsets $\Omega _{1}$
and $\Omega _{2}$. Then $\gamma$ is smooth at $P$ if and only if we can glue
two continuously differentiable functions along the curve as in~(\ref{2.1})
so that the resulting piecewise function $F$ is not differentiable at~$P$.
\end{conjecture}

\section{Supersmoothness of higher derivatives.}

Consider two non-collinear rays $v_{1}$ and $v_{2}$ emanating from the
origin in~$\mathbb{R}^{2}$. The curve formed by these two rays is not smooth
and partitions the open unit disk $\Omega $ into two sectors $\Delta _{1}$
and $\Delta _{2}$. It follows from the results of the previous section that
two differentiable functions $f_{1}$ and $f_{2}$ continuously glued along
the boundary of the sectors as in~(\ref{2.1}) produce a piecewise function 
$F_{2}$ differentiable at the origin: 
\begin{equation}
F_{2}\in C(\Omega )~\Rightarrow ~F_{2}\in C^{1}(\mathbf{0})\text{.}
\label{3.1}
\end{equation}
Farin's observation~(\ref{1}) shows that for three pairwise non-collinear
rays emanating from the origin and a piecewise function $F_{3}$ consisting
of three differentiable pieces as in Figure~\ref{circles} the following
holds:
\begin{equation*}
F_{3}\in C^{1}(\Omega )~\Rightarrow ~F_{3}\in C^{2}(\mathbf{0}).
\end{equation*}
However, as it was pointed out in the introduction, for three non-collinear
rays amplification~(\ref{3.1}) may not hold, that is, in general 
\begin{equation*}
F_{3}\in C(\Omega )~~\not\Rightarrow ~F_{3}\in C^{1}(\mathbf{0}).
\end{equation*}
In this section we extend this pattern. For a fixed $n\geq 0$, we partition
the open disk $\Omega $ into $n+2$ sectors $\Delta _{1},\dots ,\Delta _{n+2},
$ by pairwise non-collinear vectors (rays) $v_{1},\dots ,v_{n+2},$
positioned clockwise as in Figure~\ref{3.2}. Then we create a piecewise
function $F_{n+2}$ by gluing $n+2$ functions $f_{1},\ldots ,f_{n+2}\in
C^{n}(\Omega )$ along the rays. Thus for $1\leq j\leq n+1$, the sector 
$\Delta _{j}$ is formed by $v_{j}$ and $v_{j+1}$, and the sector $\Delta
_{n+2}$ is formed by $v_{n+2}$ and $v_{1}$. We will show that similarly 
to~(\ref{1}) the following holds:
\begin{equation}
F_{n+2}\in C^{n}(\Omega )~~\Rightarrow ~F_{n+2}\in C^{n+1}(\mathbf{0});
\label{mainresult}
\end{equation}
yet the weaker assumption $F_{n+2}\in C^{n-1}(\Omega )$ may not imply the
associated conclusion that $F_{n+2}\in C^{n}(\mathbf{0})$.

We start with a simple lemma that shows that two differentiable functions
continuously glued along a ray $v$ must be differentiable in the direction
of $v$. We use $D_{v}$ to denote the directional derivative in the direction
of~$v$. 

\begin{lemma}
\label{lemma}  Let $v=(a,b)$ be a unit vector in $\mathbb{R}^2$. Let $f$ and 
$g$ be continuously differentiable functions in an $\varepsilon$-neighborhood 
of the origin in $\mathbb{R}^2$ such that  
\begin{equation}  \label{3.3}
f(ta,tb)=g(ta,tb),~~\hbox{for all} ~~ t\in [0,\varepsilon ).
\end{equation}
Then  
\begin{equation*}
D_{v}f(ta,tb)=D_{v}g(ta,tb),~~\hbox{for all} ~~ t\in [0,\varepsilon ).
\end{equation*}
\end{lemma}

\begin{proof}
It suffices to prove the result for $t=0$. We obtain  
\begin{align}
D_{v}f(\mathbf{0})=& \displaystyle\lim_{t\rightarrow 0}\frac{f(ta,tb)-f(%
\mathbf{0})}{t}=\lim_{t\rightarrow 0+}\frac{f(ta,tb)-f(\mathbf{0})}{t} 
\notag \\
\overset{\text{by (\ref{3.3})}}{=} &\displaystyle\lim_{t\rightarrow 0+}\frac{%
g(ta,tb)-g(\mathbf{0})}{t} =\lim_{t\rightarrow 0}\frac{g(ta,tb)-g(\mathbf{0})%
}{t}=D_{v}g(\mathbf{0})\text{,}  \notag
\end{align}
where the second and the fourth equalities follow from the continuity of $%
D_vf$ and $D_vg$, respectively. 
\end{proof}

We are now ready to prove statement~(\ref{mainresult}). For brevity, we use $%
F:=F_{n+2}$. 

\begin{theorem}
\label{main} Let functions $f_{1},\ldots ,f_{n+2},$ be $n$ times
continuously differentiable on $\Omega $ and let $F$ be defined piecewise on
each sector $\Delta_j$ by $F\mid _{\Delta _j}:=f_{j}$, $j=1,\dots,n+2$. If 
$F\in C^{n}(\Omega )$ then $F$ has all derivatives of order $n+1$ at the
origin; that is, $F\in C^{n+1}(\mathbf{0})$, $n\geq 0$. 
\end{theorem}

\begin{proof}
If $n=0$, the proof is given in Theorem~\ref{direct}. Let $n\geq 1$. We will
show that for two neighboring functions, say $f_{j}$ and $f_{j+1}$, all
partial derivatives of order $n+1$ coincide at the origin. Then for every $k=0,\dots ,n$,  
\begin{equation*}
D_{x}^{k}D_{y}^{n-k}f_{1}(\mathbf{0})=D_{x}^{k}D_{y}^{n-k}f_{2}(\mathbf{0}
)=\ldots =D_{x}^{k}D_{y}^{n-k}f_{n+2}(\mathbf{0}),
\end{equation*}
which would prove the theorem. Without loss of generality we consider the
neighboring functions $f_{1}$ and $f_{2}$. It is clearly enough to prove
that  
\begin{equation}  \label{deriv}
D_{v_{2}}^{k}D_{v_{1}}^{n-k}f_{1}(\mathbf{0}%
)=D_{v_{2}}^{k}D_{v_{1}}^{n-k}f_{2}(\mathbf{0}),~~\hbox{for every}~~
k=0,\dots ,n.
\end{equation}
Observe that for $k\geq 1$, the assumption $F\in C^{n}(\Omega )$ implies
that the functions $D_{v_{2}}^{k-1}D_{v_{1}}^{n-k}f_{1}$ and  
$D_{v_{2}}^{k-1}D_{v_{1}}^{n-k}f_{2}$ are continuously glued along the ray~$v_{2}$. 
Hence, by Lemma~\ref{lemma} we obtain  
\begin{equation*}
D_{v_{2}}\big(D_{v_{2}}^{k-1}D_{v_{1}}^{n-k}f_{1}\big )(\mathbf{0})=D_{v_{2}}
\big(D_{v_{2}}^{k-1}D_{v_{1}}^{n-k}f_{2}\big)(\mathbf{0})
\end{equation*}
which implies~(\ref{deriv}) for $k\geq 1$. Hence it remains to prove that  
\begin{equation}  \label{derivn}
D_{v_{1}}^{n}f_{1}(\mathbf{0})=D_{v_{1}}^{n}f_{2}(\mathbf{0}).
\end{equation}
Since all the vectors $v_{j}$ are pairwise non-collinear we can find
constants $\alpha _{j}$ and $\beta _{j}$ such that $v_{1}=\alpha
_{j}v_{2}+\beta_{j}v_{j}$ for all $j=3,\dots, n+2$. Then  
\begin{align}  \label{nderiv}
D_{v_{1}}^{n} =&(\alpha _{3}D_{v_{2}}+\beta _{3}D_{v_{3}})\ldots (\alpha
_{n+2}D_{v_{2}}+\beta_{n+2}D_{v_{n+2}}) \\
&=D_{v_{2}}p(D_{v_{2}},\dots , D_{v_{n+2}})+\gamma
\prod\limits_{j=3}^{n+2}D_{v_j}  \notag
\end{align}
for some constant $\gamma$ and some homogeneous polynomial $p$ of order 
$n-1$. Since, by the assumption,  $p(D_{v_{2}},\dots, D_{v_{n+2}})f_{1}$ and $%
p(D_{v_{2}},\dots,D_{v_{n+2}})f_{2}$ coincide along the ray $v_{2}$, by
Lemma~\ref{lemma}  
\begin{equation}  \label{first}
D_{v_{2}}p(D_{v_{2}},\dots, D_{v_{n+2}})f_{1}(\mathbf{0})=
D_{v_{2}}p(D_{v_{2}},\dots,D_{v_{n+2}})f_{2}(\mathbf{0})\text{.}
\end{equation}
Similarly, for every $k=3,\ldots ,n+2$, the functions 
\begin{equation*}
{\prod\limits_{j=3, j\neq k}^{n+2}}D_{v_j}f_{k-1}~~~\hbox{and}~~~{%
\prod\limits_{j=3, j\neq k}^{n+2}}D_{v_j}f_{k}
\end{equation*}
coincide along the ray $v_{k}$. Hence, by Lemma~\ref{lemma}, for every $%
k=3,\ldots ,n+2$,  
\begin{align}
\prod\limits_{j=3}^{n+2}D_{v_j}f_{k-1}(\mathbf{0}) =&D_{v_{k}}\underset{%
j\neq k}{\prod\limits_{j=3}^{n+2}} D_{v_j}f_{k-1}(\mathbf{0}) =D_{v_{k}}
\underset{j\neq k}{\prod\limits_{j=3}^{n+2}}D_{v_j}f_{k}(\mathbf{0})
=\prod\limits_{j=3}^{n+2}D_{v_j}f_{k}(\mathbf{0})\text{.}  \notag
\end{align}
Thus, we obtain the following chain of equalities  
\begin{align}  \label{second}
\prod\limits_{j=3}^{n+2}D_{v_j}f_{2}(\mathbf{0})
=&\prod\limits_{j=3}^{n+2}D_{v_j}f_{3}(\mathbf{0})=\dots=\prod
\limits_{j=3}^{n+2}D_{v_j}f_{n+2}(\mathbf{0})
=\prod\limits_{j=3}^{n+2}D_{v_j}f_{1}(\mathbf{0})\text{.}
\end{align}
The last equality follows from Lemma~\ref{lemma} since $f_{1}$ and $f_{n+2}$
share a common edge $v_{1}$.  Thus  
\begin{align}
&D_{v_{1}}^{n} f_{2}(\mathbf{0})\overset{\text{by (\ref{nderiv})}}{=}
D_{v_{2}}p(D_{v_{2}},\dots , D_{v_{n+2}})f_2(\mathbf{0})+\gamma
\prod\limits_{j=3}^{n+2}D_{v_j}f_2(\mathbf{0})  \notag \\
& \overset{\text{by (\ref{first},\ref{second})}}{=}D_{v_{2}}p(D_{v_{2}},
\dots , D_{v_{n+2}})f_1(\mathbf{0})+\gamma
\prod\limits_{j=3}^{n+2}D_{v_j}f_1(\mathbf{0}) \overset{\text{by (\ref
{nderiv})}}{=}D_{v_{1}}^{n} f_{1}(\mathbf{0}).  \notag
\end{align}
which completes the proof of~(\ref{deriv}), and consequently proves the
theorem. 
\end{proof}

The next result  is a direct consequence of  applying Theorem~\ref{main} 
to the derivatives of the piecewise function.
\begin{corollary}
\label{supmain} Let functions $f_{1},\ldots ,f_{n+2},$ be $m$ times
continuously differentiable on $\Omega $, with $m\geq n$, and let $F_{n+2}$
be defined piecewise on each sector $\Delta _{j}$ by $F_{n+2}\mid _{\Delta
_{j}}:=f_{j}$, $j=1,\dots ,n+2$. If $F_{n+2}\in C^{m}(\Omega )$ then $F_{n+2}$ 
has all derivatives of order $m+1$ at the origin, that is, $F_{n+2}\in
C^{m+1}(\mathbf{0})$, $m\geq n\geq 0$. 
\end{corollary}

We finish this section and this article by constructing polynomials (hence
smooth functions) $f_{1},\ldots ,f_{n+2}$, $n\geq 1$,  such that the spline 
$F_{n+2}$ defined by $F_{n+2}\mid _{\Delta _{j}}=f_{j}$ is in $C^{n-1}(\Omega)$ 
yet $F_{n+2}\not\in C^{n}(\mathbf{0)}$.  We note that if $n=0$, it is
immediately obvious that $f_1\equiv 0$ and $f_2\equiv 1$ do not join
continuously at the origin.  The following observation is the key to the
construction:

\begin{lemma}
\label{simple} Given $n\geq 1$, consider the polynomial 
\begin{equation*}
g(x,y):=\sum\limits_{i=1}^{n+1}c_{i}(y+a_{i}x)^{n}\text{.}
\end{equation*}
Then the system of equations with the unknowns $(c_1,\dots,c_{n+1})$: 
\begin{equation*}
\frac{\partial ^{k}}{\partial x^{j}\partial y^{k-j}}~g(x,0)=0,~~~\hbox{for
all}~~0\leq j\leq k \leq n-1,
\end{equation*}
has a non-trivial solution.
\end{lemma}

\begin{proof}
Indeed for $0\leq k \leq n-1$ and $0\leq j\leq k$ we have 
\begin{align}
\sum\limits_{i=1}^{n+1}c_i{\frac{\partial ^{k}}{\partial x^{j}\partial
y^{k-j}}(y+a_{i}x)^{n}}\Big|_{y=0}&=\frac{n!}{(n-k)!} \sum%
\limits_{i=1}^{n+1}c_{i}a_{i}^{j}(y+a_{i}x)^{n-k}\Big|_{y=0}  \notag \\
&=\frac{n!}{(n-k)!}x^{n-k} \sum\limits_{i=1}^{n+1}c_{i}a_{i}^{n-k+j}=0. 
\notag
\end{align}
With $s:=n-k+j$, the system of $n$ equations with $n+1$ unknowns 
\begin{equation*}
\sum\limits_{i=1}^{n+1}c_{i}a_{i}^{s}=0,~~~s=0,\dots ,n-1,
\end{equation*}
has a non-trivial solution.
\end{proof}

Now we can proceed with our construction. As in Figure~\ref{3.2}, choose $n+2
$ consecutive positioned clockwise rays $v_i$ emanating from the origin
whose equations are given by the following lines $l_i$ 
\begin{equation*}
l_{1}:~y=0,~~l_{2}:~y+a_{2}x=0,~~\dots~~,~~l_{n+2}~:y+a_{n+2}x=0.
\end{equation*}
Note that without loss of generality we assume that $v_{1}$ goes along the
positive direction of the $x$-axis. Define $f_{1}:\equiv 0$ to be the
function between $v_{1}$ and $v_{2}$. Let the function between $v_{k}$ and 
$v_{k+1}$ be defined as follows:  
\begin{equation*}
f_{k}:=\sum\limits_{i=2}^{k}c_{i}l_{i}^{n},~~~\hbox{for each}~~2\leq k \leq
n+2,
\end{equation*}
with the convention $v_{n+3}:=v_1$. We next define:
\begin{equation*}
g_k(x,y):=f_{k+1}(x,y)-f_{k}(x,y)=c_{k+1}(y+a_{k+1}x)^{n},~~~\hbox{for all}
~~2\leq k \leq n+1.
\end{equation*}
All partial derivatives of $g_k$ of order $n-1$ or less vanish for 
$y=-a_{k+1}x$, that is, at the line $l_{k+1}$. It remains to choose the
coefficients $c_2, \dots,c_{n+2}$ in such a way that $f_{n+2}$ is glued
smoothly to $f_{1}\equiv 0$ at $l_{1}$, that is, so that all derivatives of
order $n-1$ or less of the polynomial 
\begin{equation*}
f_{n+2}=\sum\limits_{i=2}^{n+2}c_{i}l_{i}^{n}
\end{equation*}
vanish at $y=0$. By Lemma~\ref{simple} this leads to a system of $n$
equations with $n+1$ unknowns $(c_2,\dots,c_{n+2})$ that has a nontrivial
solution.

Hence there exists a non-zero homogeneous polynomial $f_{n+2}$ of order $n$
between $l_{n+2}$ and $l_{1}$ which is $C^{n-1}$-smoothly glued to $f_{n+1}$
across $l_{n+1}$ and $C^{n-1}$-smoothly glued to $f_{1}\equiv 0$ across 
$l_{1}$. Finally $f_{n+2}$ is a nonzero homogeneous polynomial of order $n$.
Thus there exists a partial derivative of $f_{n+2}$ of order $n$ which is a
non-zero constant. In particular, its value at the origin is not zero, yet
the same derivative of $f_{1}\equiv 0$ is zero. The resulting piecewise
function $F_{n+2}$ does not have a derivative of order $n$ at the origin.

\begin{remark}
 The existence of the spline $F_{n+2}$ implicitly constructed
above  also follows   from Theorem 9.3 in~\cite{LS}. Indeed this theorem
shows that the dimension of polynomial splines of degree $n$ and smoothness~$n-1$ 
defined over the union of $n+2$ sectors is strictly greater than $\binom{n+2}{2}$. 
The latter is the dimension of bivariate polynomials. Thus,
there exists a spline that does not have a derivative of order $n$ at the
origin. We decided to provide a development here that would be 
accessible to an audience not familiar with spline theory.
\end{remark}

\end{document}